\documentclass[12pt,a4paper]{amsart}
\usepackage{amsthm, amssymb}
\usepackage{amsfonts}
\usepackage{amsthm}
\usepackage{amsmath}
\usepackage{amscd}
\usepackage[latin2]{inputenc}
\usepackage{t1enc}
\usepackage[mathscr]{eucal}

\numberwithin{equation}{section}
\usepackage[margin=2.9cm]{geometry}
\usepackage{epstopdf} 
\usepackage{hyperref}
\newtheorem{proposition}{Proposition}[section]

 \newtheorem{definition}{Definition}[section]
\newtheorem{remark}{Remark}[section]
\theoremstyle{definition}
\newtheorem{theorem}{Theorem}[section]
\theoremstyle{plain}

\newtheorem{lemma}{Lemma}[section]

 \theoremstyle{definition}

\newtheorem{?}[Th]{Problem}

\title[]{Classification of irreducible integrable representations of loop toroidal Lie algebras}

\author[]{Priyanshu Chakraborty, Punita Batra }

\address{Priyanshu Chakraborty, Harish-Chandra Research Institute (HBNI), Chhatnag Road, Jhunsi, Prayagraj(Allahabad) 211019, Uttar Pradesh, India.}
 
\email{priyanshuchakraborty@hri.res.in}
\address{Punita Batra, Harish-Chandra Research Institute (HBNI), Chhatnag Road, Jhunsi, Prayagraj(Allahabad) 211019, Uttar Pradesh, India.}
\email{batra@hri.res.in}

\begin{document}

\begin{abstract} The aim of this paper is to classify irreducible integrable representations of loop toroidal Lie algebras with finite dimensional weight spaces. Loop toroidal Lie algebras are given by $\tau(B)=\tau_0 \otimes B \oplus D $, where $\tau_0$ is  a toroidal Lie algebra, $B$ is a commutative, associative, finitely generated unital algebra over $\mathbb C$ and $D$ is the space spanned by degree derivations $d_0,d_1,...,d_n$ over $\mathbb C$. When the center of toroidal Lie algebra acts non-trivially on modules, we prove modules come from direct sum of finitely many copies of affine Kac-Moody Lie algebras. In the case when the center of toroidal Lie algebra acts trivially on modules, we prove modules come from direct sum of finitely many copies of finite dimensional simple Lie algebras.
\end{abstract}

\maketitle



\section{Introduction} The importance of toroidal Lie algebras are well known in the theory of both mathematics and physics. Toroidal Lie algebra can be viewed as a universal central extension of the multiloop algebra $\mathfrak{\dot g} \otimes A_{n+1}$, where $\mathfrak{\dot g}$ is a finite dimensional simple Lie algebra over $\mathbb C$ and $A_{n+1}= \mathbb{C}[t_0^{\pm 1}, t_1^{\pm 1},...,t_n^{\pm 1}]$ be the laurent polynomial ring over $\mathbb C$ in $n+1$ commuting variables $t_0, t_1,...,t_n,$ for reference see \cite{1,15,16,18}. Representations of toroidal Lie algebras have been well studied in \cite {1,5,8,21}. For instance, in \cite{5} the study of irreducible modules for toroidal Lie algebras with finite dimensional weight spaces reduces to the study of irreducible modules for loop affine Lie algebras. Moreover currently mathematicians are showing interest to study loop algebras of some well known Lie algebras. For example, representations of loop algebras for finite dimensional simple Lie algebras has been studied in \cite{7,10}, reprentation for loop of Kac-Moody algebras in \cite{2,19} and loop of Virasoro algebras in \cite{20}. In the current paper we study irreducible integrable representations of loop toroidal Lie algebras with finite dimensional weight spaces, when a part of the center acts non-trivially and trivially on modules. \\

Let us consider the toroidal Lie algebra $\tau_0 = \mathfrak{\dot g} \otimes A_{n+1} \oplus Z $, $Z$ be the universal central extension of $\mathfrak{\dot g} \otimes A_{n+1}$. Define the loop toroidal Lie algebra $\tau(B)=\tau_0 \otimes B \oplus D $, where $B$ is a commutative, associative, finitely generated unital algebra over $\mathbb C$ and $D$ is the derivation space spanned by $d_0,d_1,...,d_n$ over $\mathbb C$. Lie algebra structure on $\tau(B)$ is given in Section 2. It is well known that center of the toroidal Lie algebra $\tau= \tau_0 \oplus D$, is spanned by $K_0,K_1,...,K_n,$ see \cite{1} and Section 2 for details. In the loop toroidal case center is spanned by $K_0\otimes B, K_1\otimes B,..., K_n\otimes B.$ In \cite{1,9}, S.Eswara Rao has classified irreducible integrable modules for $\tau$ with finite dimensional weight spaces with the assumptions that $K_i$, for $0\leq i \leq n$ acting on modules trivially and non-trivially. In this paper we classify irreducible integrable modules for $\tau(B)$ with finite dimensional weight spaces when $K_0\otimes 1,...,K_n\otimes 1$ acts non-trivially and trivially on modules. \\

The paper is organised as follows. In Section 2, we start by recalling definitions of toroidal Lie algebra, current Kac-Moody Lie algebra \cite{2} and define loop toroidal Lie algebra $\tau(B)$. Then we define root system, Weyl group and make a natural triangular decomposition of $\tau(B)$ as 
$$\tau(B)=\tau(B)^-\oplus \tau(B)^0\oplus \tau(B)^+. $$ 
In section 3, we define integrable modules for $\tau(B)$ and 
in Lemma \ref{2.1}, we prove an equivalent condition for integrability of $\tau(B)$ modules and $\tau$ modules. Moreover, this result distinguished our algebra from toroidal Lie algebra.
\\

 Let $V$ be an irreducible integrable module for $\tau(B)$ with finite dimensional weight spaces. In section 4, after twisting the action of $\tau(B)$ by an automorphism we assume that only $K_0\otimes 1$ acts non-trivially and all other $K_i\otimes 1$, for $1\leq i \leq n$ acts trivially on $V$ (when we consider non-trivial action). On this assumption, using the method used in \cite{1}, we find that there exists a non-zero vector $v$ such that $\tau(B)^+v=0$. Then we construct the highest weight space $T=\{ v\in V: \tau(B)^+v=0 \}  $, an irreducible module for $\tau(B)^0$. In Theorem \ref{t3.2}, we prove that weight spaces of $T$ are uniformly bounded. In Proposition \ref{p3.3}, we prove that a large part of the center acts trivially on our modules, using results of Heisenberg algebra from \cite{4}. In Theorem \ref{t3.4}, we reduce our module $V$ to an irreducible integrable highest weight module $\overline V$ for current Kac-Moody Lie algebra $ G =\mathfrak{g}_{af}'\otimes{A_n} \otimes B \oplus {\mathbb{C}d_0}$, where $\mathfrak{g}_{af}'$ is an affine Lie algebra and classify the module $\overline V$ using results of S.Eswara.Rao, P. Batra \cite{2}.\\
 
In section 5, we construct a module $\overline V \otimes A_n$ for the algebra $\widetilde G = G \oplus \widetilde D,$ $\widetilde D$ is the space spanned by the derivations $d_1,d_2,...,d_n.$ We prove that $\overline V \otimes A_n$ is completely reducible $\widetilde G$ module and $V$ is isomorphic to an irreducible component of  $\overline V \otimes A_n$ as $\widetilde G$ module.   
\\

In section 6, we assume that all $K_i\otimes 1$, for $0 \leq i \leq n$ acting on $V$ trivially and prove that whole $Z \otimes B$ acts on $V$ trivially. Then our modules becomes a module for the algebra $\mathfrak{\dot g} \otimes A_{n+1} \otimes B \oplus D$. Let $\overline \psi$ be a $\mathbb{Z}^{n+1}$ graded morphism from the universal enveloping algebra $U(\mathfrak {\dot h} \otimes A_{n+1} \otimes B) $ to $A_{n+1}$, where $\mathfrak {\dot h}$ be a fixed Cartan subalgebra of $\mathfrak{\dot g}$. Then we define universal highest weight module $M(\overline \psi)$ for  $\widetilde{\tau_1}=\mathfrak{\dot g} \otimes A_{n+1} \otimes B \oplus D$. Let $V(\overline \psi)$ be the unique irreducible quotient of $M(\overline \psi)$.
In Theorem \ref{t5.1}, we prove that any irreducible integrable module for $\tau(B)$ is isomorphic to $V(\overline{\psi})$. For each $\overline \psi$ we associate a map $\psi$ from $U(\mathfrak {\dot h} \otimes A_{n+1} \otimes B) $ to $\mathbb C$ and construct the non-graded irreducible highest weight module $V(\psi)$ for $\tau_1=\mathfrak {\dot g} \otimes A_{n+1} \otimes B$. We prove that $V(  \overline \psi)$ is isomorphic to an irreducible component of $V(\psi) \otimes A_{n+1}$. Then we prove an important result, Proposition \ref{p5.6},  related to irreducible modules for loop of simple Lie algebra. Using the help of Proposition \ref{p5.6}, we prove that $V(\psi)$ is finite dimensional. Finally we prove in Theorem \ref{t5.3}, that $V(\psi)$ is a finite dimensional irreducible module for direct sum of finitely many copies of finite dimensional simple Lie algebras, using result of \cite{10}.

\section{Notations and Preliminaries}
Throughout the paper all the vector spaces, algebras, tensor products are taken over the field of complex numbers $\mathbb{C}$. Let $\mathbb{Z}$, $\mathbb{N}$, $\mathbb{Z}_+$ denote the sets of integers, natural numbers and non-negative integers respectively.
Let $A =A_{n+1}=\mathbb{C}[t_0^{\pm1},...,t_{n}^{\pm1}]$ and 
 $A_{n}=\mathbb{C}[t_1^{\pm1},...,t_{n}^{\pm1}] $ denote the Laurent polynomial rings in $n+1$ $ (t_0,t_1,...,t_n)$ and $n$ $(t_1,t_2,...,t_n)$ commuting variables respectively.
For convenience, we denote $  \underline {m} =(m_1,..,m_n) \in \mathbb{Z}^n$, $t^{\underline {m} }=t_1^{m_1}...t_n^{m_n}$ and $(m_0,\underline m)=(m_0,m_1,...,m_n)\in \mathbb{Z}^{n+1}.$ 
For any Lie algebra $G$, let $G'$ and  $ U (G)$ denote the derived algebra $[G,G]$ and the universal enveloping algebra of $G$ respectively.\\

Let $\dot{ \mathfrak{g}}$ be a finite dimensional simple Lie algebra and  $\dot{\mathfrak{h}} $ be a Cartan subalgebra of $\dot{ \mathfrak{g}}$. Let $\dot{ \mathfrak{g}} = \dot{ \mathfrak{g}^-} \oplus {\dot{\mathfrak{h}}} \oplus \dot{ \mathfrak{g}^+}$ be the root space decomposition of $\dot{ \mathfrak{g}}$.
Let $\dot{\Delta} $ be the root system of $\dot{ \mathfrak{g}}$, $\pi=$ $ \lbrace \alpha_1 , \alpha_2,.. \alpha_d \rbrace$ be the simple roots of  $\dot{\Delta} $, $\dot{\Delta} ^+ $  $(\dot{\Delta}^-) $ be the set of positive (negative) roots of  $\dot{\Delta} $ and  $\dot \Omega$ be the Weyl group of $\mathfrak{\dot g}.$ Let $\beta $ be the maximal root of $\dot{\Delta}$. For each root $\alpha \in \dot \Delta $, let $\alpha^{\vee} \in \dot{ \mathfrak{h}} $ be such that $\alpha(\alpha^{\vee})=2$. Let $(,)$ be the normal non-degenerate symmetric bilinear form on $\dot{\mathfrak{g}}$ . Let 
 $$\dot P^+ = \{\lambda \in \mathfrak{\dot h}^*: \lambda(\alpha_i^\vee)\geq 0, \hspace{.3cm} \forall 1 \leq i \leq d  \}, $$ we call elements of $\dot P^+ $ as dominant integral weights.

Let us define $\lambda \leq \mu$ in $ \dot {\mathfrak{h}}^{*}$ iff $\mu - \lambda = \displaystyle {\sum_{i=1}^{d}}n_i \alpha_i$,    $ n_i \in \mathbb{Z}_+$.\\

Fix a commutative associative finitely generated unital algebra $B$ for this paper. Recall the definition of current Kac-Moody Lie algebra from \cite{2}.  Let $\mathfrak{g}$ be a Kac-Moody Lie algebra and $h$ be a Cartan subalgebra of $\mathfrak{g}$. Let $\mathfrak{g'} $ be the derived algebra of $ \mathfrak{g}$. Let $h'=h \cap \mathfrak{g'} $ and $h^{''}$ be the complementary subspace of $h'$ in $h$. Let $\mathfrak{g}' ={ \mathfrak{g}'}^- \oplus h' \oplus {\mathfrak{g}'}^+ $ be the standard triangular decomposition of $\mathfrak{g}'$. Let ${\mathfrak{g}}_B = \mathfrak{g}' \otimes B \oplus{h^{''}} $ and define a Lie algebra structure on ${\mathfrak{g}}_B$ by 
$$ [X\otimes b, Y \otimes b']=[X,Y]\otimes bb', $$ 
 $$[X\otimes b , h]= [X,h]\otimes b , $$ $\forall$ $X,Y \in {\mathfrak{g}}$, $b,b' \in B , h \in h''.$ This algebra is called current Kac-Moody Lie algebra.\\

Let us consider the multi-loop algebra $\mathfrak{\dot g} \otimes A$ with the usual Lie bracket.
Let 
      $$\Omega _A =span\{t_0^{k_0}t^{\underline k}K_i: 0\leq i\leq  n,  k_0 \in \mathbb{Z},  \underline k \in \mathbb{Z}^n  \}.$$ 
      Also let $dA$ be the subspace of $\Omega_A$ defined by  $$ dA= span\{\displaystyle{\sum_{i=0}^{n}}  k_i t_0^{k_0} t^{\underline k} K_i:  k_0 \in \mathbb{Z},  \underline k \in \mathbb{Z}^n   \}.$$ 
 Let us consider $Z = \Omega_A / dA$ and
 $\tau_0 = \dot{ \mathfrak{g}} \otimes A \oplus Z  $. Lie algebra structure on $\tau_0$ is defined by\\

1. $\left[ x \otimes t_0^{k_0} t^{\underline k}, y\otimes  t_0^{r_0} t^{\underline r} \right] = [x,y]\otimes t_0^{k_0+r_0} t^{\underline k+\underline r} + (x,y)\displaystyle{\sum_{i=0}^{n} k_i t_0^{k_0+r_0} t^{\underline k+ \underline r} K_i},$\\ \hspace*{.8cm}    $\forall x,y \in \mathfrak{\dot g}, k_0,r_0 \in \mathbb Z, \underline{k}, \underline{r} \in \mathbb Z^n.  $\\

2. $ Z$ is central in $\dot{\mathfrak{g}} \otimes A \oplus Z$. \\ 

$\tau_0$ is called toroidal Lie algebra.
It is well known that $\tau_0$ is the universal central extension of  $\dot{\mathfrak{g}} \otimes A $; for more details see \cite{15,16}.  \\

Let $\tau (B) = \tau_0\otimes B \oplus D $ and define a Lie algebra structure on $\tau(B)$ by                \\ 
 
1. $[x\otimes b, y\otimes b']=[x,y]\otimes bb',$ \hspace{.5cm}  $\forall x,y \in \tau_0, b,b' \in B.$\\

2. $[d_i, x \otimes t_0^{k_0} t^{\underline k} \otimes b]=k_i x \otimes t_0^{k_0} t^{\underline k}\otimes b, $ \hspace{.5cm} $\forall$  $   0\leq i\leq  n , b \in B, x \in \mathfrak{\dot g},  k_0 \in \mathbb Z, \underline{k} \in \mathbb Z^n $\\

3.  $[d_i, t_0^{k_0} t^{\underline k}K_j \otimes b]=k_i t_0^{k_0} t^{\underline k} K_j \otimes b$, \hspace{.5cm} $\forall$  $  0\leq i, j \leq  n , b \in B, k_0 \in \mathbb Z, \underline{k} \in \mathbb Z^n .$
\\

4. $\left[ d_i, d_j \right] = 0$, \hspace{.5cm} $\forall$         $ 0 \leq i, j \leq n.$\\

We call $\tau(B)$ as loop toroidal Lie algebra.
\\ 

 Let $ {H}= \mathfrak{\dot h} \bigoplus \displaystyle{ (\bigoplus_ {i=0}^{n }\mathbb{C}K_i)}\bigoplus \displaystyle{ (\bigoplus_ {i=0}^{n }\mathbb{C}d_i)} $, where we have identified $\mathfrak{\dot h}\otimes 1 \otimes 1 , K_i\otimes 1$ with $\mathfrak{\dot h} , K_i$ respectively.  
Then H is a Cartan sub-algebra of $\tau(B)$ . Let $\delta_i, \omega  _i \in H^*$ be such that

 $$ \delta _i (\mathfrak{\dot h}) =0, \hspace{.5cm} \delta _i (K_j)=0 , \hspace{.5cm} \delta _i (d_j)= \delta_{ij} ,  \hspace{.5cm} 0 \leq i,j \leq n ,$$ 
 
  $$ \omega_i (\mathfrak{\dot h}) =0  , \hspace{.5cm} \omega _i (K_j)= \delta _{ij} , \hspace{.5cm} \omega_i (d_j)=0 ,  \hspace{.5cm} 0 \leq i,j \leq n . $$ \\
 For $  \underline {m} =(m_1,..,m_n) \in \mathbb{Z}^n $,  let  $ \delta_ {\underline m}= $ $ \displaystyle \sum_{i=1}^{n} m_i \delta_i $. Then $ \tau(B) $ has root space decomposition with respect to H as  $$ {\tau(B)}= \tau(B)_0 \oplus (\displaystyle\bigoplus_{\alpha \in \Delta}\tau(B)_{\alpha}  ),$$   \\
where $\Delta = \dot{\Delta} \cup \{ \alpha+ m_0 \delta_0 + \delta_{\underline m} : \alpha \in \dot\Delta \cup \{0 \},    \underline {m}  \in \mathbb{Z}^n, m_0 \in  \mathbb{Z}, \hspace{.2cm} (m_0, \underline m) \neq (0,0)  \}$ and
     $$ \tau(B)_{\alpha+ m_0 \delta_0 + \delta_{\underline m} }= {\dot{\mathfrak{g}}_\alpha} \otimes t_0^{m_0} t^{\underline m}\otimes B,$$  
    $$ \tau(B)_{ m_0 \delta_0 + \delta_{\underline m} }= {\dot{\mathfrak{h}}} \otimes t_0^{m_0} t^{\underline m}\otimes B \oplus( \bigoplus_{i=0}^{n} {\mathbb{C}t_0^{m_0} t^{\underline m}K_i} ) \otimes B,$$ 
    $$ \tau(B)_0=    \dot{\mathfrak{h}}\otimes B \oplus {\bigoplus_{i=0}^{n}\mathbb{C}K
    _i\otimes B \oplus {\bigoplus_{i=0}^{n}\mathbb{C}d_i} .}$$
    
Let,      $$\tau(B)^{+} ={ {\dot {\mathfrak{g}}^{+}} \otimes \mathbb{C} [t_1^{\pm1},...,t_{n}^{\pm1}]}\otimes B  \oplus  {\dot{\mathfrak g}\otimes {t_0{ \mathbb{C}[t_0,t_1^{\pm1},...,t_{n}^{\pm 1}]}} \otimes B} \oplus{ \displaystyle{\bigoplus_{i=0}^{n} }}  {t_0{ \mathbb{C}[t_0,t_1^{\pm1},...,t_{n}^{\pm 1}]}}K_i \otimes B  , $$ 
 
     $$\tau(B)^{-} ={ {\dot {\mathfrak{g}}^{-}} \otimes \mathbb{C}[t_1^{\pm1},...,t_{n}^{\pm1}]}\otimes B \oplus  {\dot{\mathfrak g}\otimes {t_0^{-1}{ \mathbb{C}[t_0^{-1},t_1^{\pm1},...,t_{n}^{\pm1}]}}\otimes B} \oplus{ \displaystyle{\bigoplus_{i=0}^{n} }}  {t_0^{-1}{ \mathbb{C}[t_0^{-1},t_1^{\pm1},...,t_{n}^{\pm 1}]}}K_i \otimes B  ,$$  
  
     $$\tau(B)^{0} ={ {\dot {\mathfrak{h}}} \otimes \mathbb{C}[t_1^{\pm1},...,t_{n}^{\pm1}]\otimes B} \oplus{ \displaystyle{\bigoplus_{i=0}^{n} }}  { { \mathbb{C}[t_1^{\pm1},...,t_{n}^{\pm 1}]}}K_i \otimes B 
    \oplus (\displaystyle {\bigoplus_{i=0}^{n}} {\mathbb{C}d_i} ). $$ 
   
    Then $\tau(B)$ has a natural triangular decomposition as $$\tau(B) = \tau(B)^- \oplus \tau(B)^0 \oplus \tau(B)^+ .$$ 
\\
Now extend $ \alpha \in{\dot {\Delta}} $ to an element in $ H^* $ by defining $\alpha(K_i)=0$,  $\alpha(d_i)=0 $, for         $ 0 \leq i \leq n $  and the non-degenerate symmetric bilinear form of $\mathfrak{\dot h}$ to a bilinear form on $H^*$ by 

 $$ (\alpha_i , \delta_k)=(\alpha_i, \omega_k)=0  ,\hspace{.2cm} 1 \leq i \leq d , 0 \leq k \leq n ,$$ 
 
  $$ ( \delta_k,\delta_p)=(\omega_k, \omega_p)=0 ,  \hspace{.5cm}(\omega_k,\delta_p)=\delta_{kp}, \hspace{.2cm  }0 \leq k,p \leq n . $$ 
  \\ We call  $ \gamma \in \Delta $ is a real root if $(\gamma, \gamma) \neq 0$ and denote the set of real roots as $\Delta^{re}$. For $\gamma=\alpha + m_0\delta_0 + \delta_{\underline m} \in \Delta $, define
   $$  \gamma^{\vee} = \alpha^{\vee} + \frac{2}{(\alpha, \alpha)}{\displaystyle \sum_{i=0}^{n}}m_iK_i, $$
    where $\alpha^{\vee} \in \dot{ \mathfrak {h}} $ has the property that $\alpha(\alpha^{\vee})=2$. Also check that $ \gamma(\gamma^{\vee})=2$.\\ For a real root $\gamma \in \Delta$ define a reflection $r_\gamma$ on $ H^*$ by $$ r_{\gamma}(\lambda)= \lambda - \lambda(\gamma^{\vee}) \gamma , \hspace{.2cm}\lambda \in H^* .$$
Let $ \ {\Omega} $ be the Weyl group generated by the reflections corresponding to real roots of $\Delta $.   
\section{Integrable $\tau(B)$ modules}
\begin{definition}\label{2.1}
A module $V$ for $\tau(B)$ is said to be integrable if \\

1. If the action of $H$ on $V$ is diagonalizable, i.e $V $ can be decomposed as $ V=\displaystyle { \bigoplus_{\lambda \in H^*}}{V_{\lambda}}$,
where $ V_{\lambda} =\{v \in V: hv=\lambda(h)v , \forall h \in H \}.$  $ V_{\lambda}$'s are called weight spaces of $V$ of weight $\lambda$.\\

2. If for all $\alpha \in \dot{\Delta}$, $(m_0, \underline{m}) \in \mathbb{Z}^{n+1}$, $b \in B$ and $v \in V$  there exists an integer k $ =k(\alpha,(m_0,\underline m),b,v)$ such that $(X_{\alpha}\otimes t_0^{m_0}t^{\underline m}\otimes b)^k.v =0  $, where $X_{\alpha} $ is the root vector corresponding to the root $\alpha \in \dot{\Delta}$.
\end{definition}
\begin{lemma}\label{l2.1}
Let $V$ be a module for $\tau(B)$ with finite dimensional weight spaces with respect to $H$. Then $V$ is integrable for $\tau(B)$ if and only if it is integrable for $\tau$ $($see the definition of integrable modules for $\tau$ in Section 2 of \cite{1}$)$.
\end{lemma}
\begin{proof}
It is immediate that $V$ is integrable $\tau$ module from Definition \ref{2.1} and the definition of integrable module for $\tau$.   \\
Conversely assume that $V$ is $\tau$ integrable. Weight spaces of $V$ as $\tau$ and $\tau(B)$ modules are same, since both have same Cartan sub-algebra. 
Consider  $ V'=\displaystyle { \bigoplus_{k \in \mathbb N}}{V_{\lambda + k\gamma}}$, where $\lambda$ is a fixed weight and $\gamma=\alpha + m_0\delta_0 +\delta_{\underline m}$ is a fixed real root.\\
{\bf Claim:} $V'$ is finite dimensional. Suppose not, let $k_1$ be the smallest positive integer such that $V_{\lambda + k_1\gamma} \neq 0$, pickup $v_1'(\neq 0)  \in V_{\lambda + k_1\gamma}.$ As $V$ is integrability  $\tau$-module there exists a natural number $r_1$ such  $(X_{\alpha}\otimes t_0^{m_0}t^{\underline m})^{r_1}.v_1' =0 $ but $(X_{\alpha}\otimes t_0^{m_0}t^{\underline m})^{r_1-1}.v_1' \neq 0.$\\ Consider $v_1=(X_{\alpha}\otimes t_0^{m_0}t^{\underline m})^{r_1-1}.v_1' \in V_{\lambda + (k_1+r_1-1)\gamma} =  V_{\lambda + n_1\gamma} $      $($ say $)$, then $(X_{\alpha}\otimes t_0^{m_0}t^{\underline m}).v_1 =0 $. Since $V'$ is not finite dimensional there exists a $n_2>n_1$ such that $v_2 \in  V_{\lambda + n_2\gamma}$ and  $(X_{\alpha}\otimes t_0^{m_0}t^{\underline m}).v_2 =0 $. Continuing this we can construct an increasing sequence $\{ n_i \}$ of natural numbers and $v_i \in V_{\lambda + n_i\gamma} $ such that 
$(X_{\alpha}\otimes t_0^{m_0}t^{\underline m}).v_i =0. $ \\Consider the $sl_2$ copy of the toroidal Lie algebra $S_{\gamma}=$ span $ \{ X_{\alpha}\otimes t_0^{m_0}t^{\underline m}, X_{-\alpha}\otimes t_0^{-m_0}t^{-\underline m}, \gamma^\vee \}$. Due to integrability and by choices of $v_i$, $U(S_\gamma)v_i$'s are finite dimensional highest weight modules for all $i$. Let $V_i$ be the irreducible component of $U(S_\gamma)v_i$ containing the highest weight vector $v_i$ with highest weight $\lambda(\gamma^\vee)+ 2n_i $. Then the sum of $V_i$ is direct.
Consider $N_0 \in \mathbb Z$ such that $\lambda(\gamma^\vee) + 2N_0 \geq 0$, this is possible since $\lambda(\gamma^\vee) \in \mathbb Z,$ Lemma 2.3 of \cite{1}. Then by $\mathfrak{sl}_2$ theory $V_{\lambda + N_0\gamma}$ is not finite dimensional, since sum of $V_i$'s is direct and sequence $\{ n_i \}$ is increasing. Hence the claim.
 \\Therefore there are only finitely many weights of $V$ as $\tau(B) $ module of the form ${\lambda + k\gamma}$ such that $k \in \mathbb N.$ Thus $V$ is $\tau(B)$ integrable.
     
\end{proof}

\begin{lemma}\label{l2.3}
Let $V$ be an irreducible integrable module for $\tau(B)$ with finite dimensional weight spaces with respect to $H$. Then,\\ 
1. $P(V)$ is  $\Omega$-invariant, $P(V)$ is the set of all $\lambda \in H^* $ such that $V_\lambda \neq 0$   . \\
2. $dim V_\lambda = dim V_{\omega_{ \lambda}} $ , $\forall \omega \in \Omega , \lambda \in P(V)$. \\
3.  If $\alpha \in \Delta^{re} , \lambda \in P(V) $ then $\lambda(\alpha^{\vee}) \in \mathbb{Z}$.\\
$4$.  If $\alpha \in \Delta^{re} , \lambda \in P(V) $ and $ \lambda(\alpha^{\vee})>0(<0)$ then $\lambda - \alpha(\lambda+\alpha) \in P(V)$.\\
5. $ \lambda{(K_i \otimes 1)} $ is a constant integer for all $\lambda\in P(V)$. 
\end{lemma}
\begin{proof}
 Note that in Lemma 2.3 of \cite{1} for the statements 1 to 4 only integrability is required and use Lemma \ref{l2.1}. Since $K_i\otimes b$ are cental in $\tau(B)$, so $\lambda(K_i\otimes b) $ acts on $V$ as saclar. It follows from Lemma 2.3 of \cite{1} that $ \lambda{(K_i \otimes 1)} \in \mathbb{Z}$.
\end{proof}
 
\section{The center of toroidal Lie algebra acts non-trivially }
In this section we will reduce our module to a module for current Kac-Moody Lie algebra, when $K_i$, for $0\leq i \leq n$ acts non-trivially on our modules.\\

Let $GL(n+1,\mathbb{Z}) $ be the set of all $n+1 \times n+1$  matrices with determinant $\pm 1 $. There is a natural action of $GL(n+1,\mathbb{Z}) $ on $\mathbb{Z}^{n+1} $.  For every $ P \in GL(n+1,\mathbb{Z})$ define $P:\tau(B) \to \tau(B)$  by
     $$ P( X \otimes{t^{\underline m}} \otimes b )  = X \otimes{t^{P \underline m}} \otimes b, $$  
    $$ P(d( t^{\underline m}) t^{\underline r} \otimes b)= d( t^{P \underline m}) t^{P \underline r} \otimes b ,$$
      for all b in B and $\underline m \in \mathbb{Z}^{n+1} $.\\
 
Let $(d_0', d_1',..,d_n')=(P^T)^{-1}(d_0,d_1, ..,d_n)$ and define $P(d_i)= d_i' $.\\

It is easy to check that $P$ is an automorphism of $\tau(B)$.
If $P=(p_{ij})_{0 \leq i,j \leq n}$, then  $P(t^{\underline m} K_j \otimes b) = \displaystyle \sum_{i=0}^{n} { p_{ij}t^{P \underline m} K_i \otimes b }$. Hence taking $\underline m =0 $ we get

\begin{align}\label{e3.1}
 P(K_j \otimes b) &= \displaystyle \sum_{i=0}^{n} { p_{ij} K_i \otimes b } .
\end{align}

\begin{definition}
A linear map $z : V \to V$ is said to be a central operator of degree $(m_0,\underline m)$ for $\tau(B)$ if $z$ commutes with $\dot{\mathfrak g}\otimes A \otimes B \oplus Z\otimes B$ and $d_iz -zd_i= m_iz$ for all $0 \leq i \leq n$.
\end{definition}
 
\begin{lemma}\label{l3.1}$($Lemma 1.7,1.8 \cite{5}$)$
Let $V$ be an irreducible module for $\tau(B)$ with finite dimensional weight spaces with respect to $H$. Then, \\
1. Let $z$ be a central operator of degree $(m_0, \underline m)$ such that $zu \neq 0$ for some $u \in V.$ Then $zw \neq 0$ for all non-zero $w \in V$.\\
2. Let $z$ be a non zero central operator of degree  $(m_0, \underline m)$. Then there exists a operator $T$ of degree  $(-m_0, -\underline m)$ such that  $zT=Tz=Id$ on $V$.\\
3. Up to a scalar multiple there is at most one non-zero central operator of a given degree.

\end{lemma}

\begin{theorem}\label{t3.1}
Let $V$ be an irreducible module for $\tau(B)$ with finite dimensional weight spaces with respect to $H$ and let $ L =\{ (r_0,\underline r) \in \mathbb{Z}^{n+1} : t_0^{r_0}t^{\underline r} K_i \otimes b \neq 0 $ on $V$ for some $ 0\leq i \leq n , b \in B\} $. Let $S$ be the subgroup of $\mathbb{Z}^{n+1}$ generated by L and assume that rank of S be k. Then upto an automorphism of $\tau(B)$
 \\1. there exists non zero central operators $ z_i$ for all  $ i \geq n-k+1 $ such that degree of $z_i$ is equal to $(0,0,...,l_i,0,..0 ). $
 \\ 2. $k < n+1$. 
 \\3.  $t_0^{r_0}t^{\underline r} K_i \otimes b = 0 $ on $V ,$ for all $ r_j $, $n-k+1 \leq i \leq n $ and $b \in B.$
 \\4. $t_0^{r_0}t^{\underline r} K_i \otimes b = 0 $ on $V ,$ for all $ r_j \neq 0 $, $ 0 \leq j \leq n-k $ and for all $0\leq i \leq n$, $b \in B.$
\\5. There exists a proper submodule $W$ of $V$ for $ \dot{\mathfrak{g}} \otimes A \otimes B \oplus  Z \otimes B \oplus D_k $ such that $V/W$ has finite dimensional weight spaces  with respect to $ \mathfrak{\dot h}\oplus \displaystyle{ \bigoplus_{i=0}^{n}\mathbb{C}K_i }\oplus D_k$, where $D_k $ is the space spanned by $ d_0, d_1, .. ,d_{n-k}$. 
\end{theorem}
\begin{proof}
Proof of this theorem follows on the similar lines as in  Theorem 4.5 of \cite{1}.
\end{proof}
\begin{lemma}\label{ln4.2}
Let $(c_1,...,c_n)$ be a given non-zero vector in $\mathbb Z^n$. Then there exists a  $P\in GL(n,\mathbb Z)$ such that $P.(c_1,...,c_n)^T=(c_1',0,0,...0)^T$ and $c_1'$ is a positive integer.  
\end{lemma}
\begin{proof}
We  prove this lemma by induction on $n$. Let $(c_1,c_2)$ be a non-zero vector in $\mathbb Z^2$ and $d=gcd(c_1,c_2)$. Then there exists $p,q \in \mathbb Z$ such that $pc_1+qc_2=d$. Now consider the matrix,
  $$P= \begin{pmatrix}
p & q \\
\frac{-c_2}{d} & \frac{c_1}{d} \\ 
\end{pmatrix}  $$
It is easy to see that $P$ is our required matrix.\\
Assume that the result is true for $n=k.$ Let $(c_1,c_2,...,c_{k+1})$ be a non-zero vector in $\mathbb Z^{k+1}$. Then we can find two matrices $P_1$ and $P_2$ in $GL(n, \mathbb Z)$ such that $P_1(c_2,c_3,...,c_{k+1})^T=(c_2',0,...,0)^T$ and $P_2(c_1,c_2')^T=(c_1',0)^T.$ Then consider 
 $$ P= \begin{pmatrix}
 P_2& 0 \\
0&I_{k-1\times k-1} \\ 
\end{pmatrix}
 \begin{pmatrix}
1 & 0 \\
0 & P_1 \\ 
\end{pmatrix}    $$
Hence the result.

\end{proof}

 Assume that $K_i $ acting on $V$ as $c_i $ for $0\leq i\ \leq n$ and not all $c_i$ are zero. By Theorem \ref{t3.1}, we have $ c_n =c_{n-1} =..=c_{n-k+1}=0 $ upto an automorphism of $\tau(B)$, if the rank of S is $k$. By Lemma \ref{l2.3}, $c_i$'s are integer. Now using Lemma \ref{ln4.2} we can find a matrix $P_1\in GL(n-k+1, \mathbb{Z})$ such that $P_1(c_0,c_1,..,c_{n-k})=(c_0',0,0,..0) $ and $c_0'$ is a positive integer. Consider

\begin{equation*}
P= 
\begin{pmatrix}
P_1 & 0  \\
0 & I \\ 
\end{pmatrix}
\end{equation*}
where $I$ is the identity matrix of order $ k \times k$.  Thus using \ref{e3.1}, we can assume that only $K_0$ acts as a positive integer and all other $K_i$'s acts trivially on V upto an automorphism of $\tau(B)$.

\begin{proposition}\label{p3.1}
Let $V$ be an irreducible integrable module for $\tau(B)$ with finite dimensional weight spaces with respect to $H$. Also assume that $K_0$ acts by a positive integer $c_0$ and $K_i$ acts trivially on $V$, for $1 \leq i \leq n$. Then there exists a non-zero weight vector $v$ in $V$ such that $\tau(B)^{+}.v = 0$.
\end{proposition} 
The Proof of Proposition \ref{p3.1} runs parallelly to the method used in \cite{1} to prove Proposition 2.4. Note that $\tau(B)^+.v=0$ implies that, if $v \in V_\lambda$ then $\lambda+\alpha \notin P(V)$ for all $\alpha\in \dot \Delta$. In particular, from \ref{l2.3}(4) we have, $\lambda(\alpha^\vee) \geq 0$, i.e $\lambda \in \dot{P^+}$. We call elements of $V_\lambda$ as highest weight vectors.

\begin{lemma}\label{ln4.3}
Let us consider $T=\{v \in V : \tau(B)^{+}v = 0 \}.$ Then\\
1. $T$ is an irreducible weight module for $\tau(B)^0$ with finite dimensional weight spaces.\\
2. Weights of $T$ are lies in the set  $ \{  \lambda + \delta_{\underline m} : \delta_{\underline m}= \displaystyle{\sum_{i=1}^{n}m_i \delta_{i}} ,  {\underline m} =(m_1,m_2,..,m_n) \in \mathbb{Z}^n  \}$ for some fixed weight $\lambda$ of $T$. In particular, for every weight $\mu$ of $T$, $\mu(\alpha_i^\vee)\geq 0$, for $1\leq i \leq d$.
\end{lemma}
\begin{proof}
It is easy to see that $T$ is $\tau(B)^0$ invariant and by Proposition \ref{p3.1}, $T \neq 0 .$ Moreover, viewing $T$ as $\tau(B)^0$ sub-module of $V$ and using the fact that sub-module of a weight module is weight module we have, 
  $ T=\displaystyle { \bigoplus_{\lambda \in H^*}}{V_{\lambda}} \cap T. $ Hence weight spaces of $T$ are finite dimensional.\\
Let $w$ be a non-zero vector of $T$ and $v$ be a non-zero weight vector of $T$ of weight $\lambda$ (say). Since $V$ is irreducible there exists a non-zero $X,Y \in U(\tau(B))$ such that $X.w=v$ and $Y.v=w$, where $X=X^-HX^+$ and $Y=Y^-H'Y^+,$ for some  $X^-, Y^-\in \tau(B)^-,H,H'\in \tau(B)^0, X^+,Y^+ \in \tau(B)^+$. Suppose $Y$ contains the $Y^-$ part in the expression, then $w$ will be a weight vector of weight $\lambda-\sum{\alpha} + \delta_{\underline{m}}$ for some roots $\alpha$ of $\mathfrak{\dot g} $, which is absurd by $X.w=v$. Thus $H^{'}.v=w$, which clearly implies that $X$ cannot contain $X^-$ part. Hence $T$ is an irreducible $\tau(B)^0$-module. Note that second part of the lemma immediately follows from 1.
  
\end{proof}

Now we are going to prove that the weight spaces of $T$ are uniformly bounded. Similar result has been proved in case of twisted toroidal extended affine Lie algebras in \cite{3}. For $1 \leq i \leq n$, consider the loop algebra $G_i= {\dot {\mathfrak g}} \otimes{\mathbb{C}[t_i^{\pm1}]} \oplus \mathbb{C}d_i$ and the lattice $M_i$ generated by $$\{t_{i,h}: h \in \mathbb{Z}[\dot{\Omega}(\beta^\vee)] \},$$ where $t_{i,h}:(\mathfrak{\dot h}\oplus \mathbb C d_i)^* \to (\mathfrak{\dot h}\oplus \mathbb C d_i)^*$ defined by
$t_{i,h}(\lambda)=\lambda-\lambda(h)\delta_i$. Note that ${t_{i,h}}^{-1}(\lambda)=\lambda+\lambda(h)\delta_i$.\\
Let $\Omega_i$ be the Weyl group of $G_i$. It is clear that $\Omega_i \subset \Omega$. Moreover, It is well known that Weyl group $\Omega_i$ is semi direct product of $M_i$ and $ \dot{\Omega}$, see [Proposition 6.5, \cite{13}] for more details.\\

\begin{lemma}\label{ln4.4}
Let $\lambda$ be a non-zero weight of $T$.\\
1. Either weight spaces of $T$ are one dimensional or $\lambda|_{\mathfrak{\dot h}}\neq 0$.\\
2. If $\lambda|_{\mathfrak{\dot h}}\neq 0$, then there exists a $r_\lambda \in \mathbb N$ and $w \in {\Omega}$ such that $w({{\lambda}} +\delta_{\underline r} )={\lambda}+ \displaystyle{\sum_{i=1}^{n}s_i \delta_i}$ such that $0 \leq s_i <r_{{\lambda}}.$

\end{lemma}
\begin{proof}
1. Let us assume $\lambda|_{\mathfrak{\dot h}}= 0$ and $v_\lambda$ be a non-zero weight vector of $T$. We assert that $ \mathfrak{\dot g}^-\otimes A_n \otimes B$ acts trivially on $v_\lambda$. If not, suppose $y_\alpha\otimes t^{\underline{m}} \otimes b$ acts non-trivially on $v_\lambda$. Then $\lambda-\alpha+\delta_{\underline{m}}$ will be a non-zero weight of $V$ and hence by Lemma \ref{l2.3}(2)
   $$ r_\alpha(\lambda-\alpha+\delta_{\underline{m}})=\lambda+\alpha+\delta_{\underline{m}} , $$ 
   will be a non-zero weight of $V$, a contradiction. Therefore $\mathfrak{\dot h}\otimes A_n\otimes B$ acts trivially on $v_\lambda$. Hence $W=\{v \in T:\mathfrak{\dot h}\otimes A_n\otimes B.v=0\}$ is non-zero sub-module of $T$. Therefore by irreducibility of $T$, we have $W=T$. Again, it is easy to see that as a Lie algebra $\mathfrak{\dot h}\otimes A_n\otimes B$ generates ${ \displaystyle{\bigoplus_{i=0}^{n} }}  { { \mathbb{C}[t_1^{\pm1},...,t_{n}^{\pm 1}]}}K_i \otimes B $. Therefore $T$ is an irreducible module for the abelian Lie algebra $\displaystyle {\bigoplus_{i=0}^{n}} {\mathbb{C}d_i} $. Now the result follows from the fact that each weight space of $T$ is invariant under $\displaystyle {\bigoplus_{i=0}^{n}} {\mathbb{C}d_i} $. \\
   2.  Let $r_{{\lambda}}=$min $\{$ $ {{\lambda}}(h) :{{\lambda}(h)} >0 , h \in \mathbb{Z}[\dot{\Omega}(\beta^\vee)]  \}$. This $r_{{\lambda}}$ exists and  $r_{{\lambda}} \in \mathbb N$, since $\mathbb{Z}[\dot{\Omega}(\beta^\vee)]$ is contained in the $ \mathbb{Z}$ linear span of ${\alpha_i}^\vee$ and $\lambda(\alpha_i^\vee)>0$ by Lemma \ref{ln4.3}(2), for $1\leq i \leq d. $ Let $r_{{\lambda}}= \lambda (h_0)$ for some $h_0 \in \mathbb{Z}[\dot{\Omega}(\beta^\vee)].$\\
 For each $1 \leq i \leq n$, consider ${{\lambda}} + r_i\delta_i$ . Then by Euclidean algorithm we have, $r_i=k_ir_{{\lambda}} + R_i$ for some $0 \leq  R_i<r_{{\lambda}}$ and $k_i \in\mathbb Z.$ Now, $${t_{i,h_0}}^{ k_i}( \lambda +  r_i\delta_i )= \lambda +(r_i- k_i{\lambda}(h_0))\delta_i= \lambda +R_i\delta_i.$$ 
   Therefore we have 
  $${\displaystyle{ \prod _{i=1}^{n}}{t_{i,h_0}}^{ k_i}}  ({{\lambda}} +\delta_{\underline r} )={\lambda}+ \displaystyle{\sum_{i=1}^{n}R_i \delta_i},$$ where $\displaystyle{ \prod _{i=1}^{n}}{t_{i,h_0}}^{ k_i}\in \displaystyle{ \prod _{i=1}^{n}{\Omega_i}}\subset \Omega.$
\end{proof}

\begin{theorem}\label{t3.2}
The dimensions of the weight spaces of $T$ are uniformly bounded.
\end{theorem} 
\begin{proof}
Consider the set $S_1 =\{ \mu: \mu ={\lambda + \displaystyle{\sum_{i=1}^{n}s_i \delta_i}}$, $0 \leq s_i < r_{{\lambda}}   \}$. This is a finite set and hence consider $N_0 =Max \{ dimT_{\mu}: \mu \in S_1 \}$. Then by Lemma \ref{l2.3}(2), Lemma \ref{ln4.3}(2) and Lemma \ref{ln4.4}, dimensions of weight spaces of $T$ are bounded by $N_0$.

\end{proof}

Now we are going to prove that a large part of the center acts trivially on our module. Before that we record a result from \cite{4}. 
\begin{proposition} \label{p3.2}
Let $\widetilde H$ be a Heisenberg Lie algebra consisting of basis $\{a_k:k \in \mathbb Z-\{0\} \}\cup\{c\}$ with multiplication 
$$[a_i,a_{-j}]=ic\delta_{ij}, \hspace{1cm} [a_i,c]=0 $$
for all $i,j \in \mathbb Z-\{0\}.$\\
Let $M=\displaystyle{\bigoplus_{i\in \mathbb Z}M_i}$ be a non-trivial $\mathbb Z$ graded $\widetilde H$ module on which $a_k$ acts an operator of degree k $($i.e $a_kM_i\subseteq M_{i+k})$ and $c$ acts as a non-zero scalar. Then dim$M_i's$ are not uniformly bounded.
\end{proposition}

\begin{proposition}\label{p3.3}
Let $V$ be a module for $\tau(B)$ satisfying the conditions of Proposition \ref{p3.1}. Then $ K_i\otimes b $ acts on $V$ trivially for $1\leq i \leq n$ and for all $ b \in B.$
\end{proposition}
\begin{proof}
Fix $1\leq i \leq n$, $ b \in B$ and assume that 
$ K_i\otimes b $ acts on $V$ non-trivially. Let $ h, h' \in \mathfrak{ \dot h} $ be such that  $(h,h')\neq 0$ (this is possible due to the non-degenerate bilinear form on $\mathfrak{ \dot h}$). Consider the Heisenberg algebra $\widetilde{H}=$ span $ \{h\otimes{t_i^{l}}\otimes b , h'\otimes{t_i^{-l}}\otimes 1 , K_i\otimes b , l \in \mathbb{N} \}$ with the bracket 

$$[h\otimes{t_i^{n}}\otimes b, h'\otimes{t_i^{-m}}\otimes 1]=(h,h')n\delta_{nm}K_i\otimes b   $$ 
 and $K_i \otimes b$ is central in $\widetilde H.$\\
Let $M =U(\widetilde H)v_{\lambda} $, where $v_{\lambda}$ is a weight vector in $T$. Then $M = \displaystyle{\bigoplus_{r \in \mathbb{Z}}}M_r$ , where $M_r = M \cap T_{\lambda-r \delta_i}$, hence $dimM_r $'s are uniformly bounded by Theorem \ref{t3.2}, a contradiction to Proposition \ref{p3.2}.
\end{proof}
\begin{theorem}\label{t3.3}
Let $V$ be a module for $\tau(B)$ satisfying properties of Proposition \ref{p3.1}. Then the rank of $S$ (constructed in Theorem \ref{t3.1}) is $n$.
\end{theorem}
\begin{proof}
Let us assume that rank of $S$ be $k<n $.
By Proposition \ref{p3.1}, there exists a weight vector $v \in V$, say weight $\lambda$ such that $ \tau(B)^+v =0$. In particular, $h \otimes t_0^m t_{n-k}^r \otimes b v=0$ for all $h \in \mathfrak{\dot h}$,  $m>0 $ , $r \in \mathbb{Z}$ and $  b \in B$.  For $h(\neq 0) \in \mathfrak{\dot h}$, consider the following set of vectors

 $$S'=\{ h\otimes t_0^{-m} t_{n-k}^{r-d}\otimes b.h\otimes t_0^{-(m+1)} t_{n-k}^{d}\otimes b'.v : d \in \mathbb{Z} \}, $$ for some fixed $ m >0 , r \in \mathbb{Z} , b , b' \in B.$ \\
{\bf Claim:} The set $S'$ is linearly independent.
 Consider 
 \begin{align}
  \displaystyle\sum_{d \in \mathbb{Z}} a_d h\otimes t_0^{-m} t_{n-k}^{r-d}\otimes b.h\otimes t_0^{-(m+1)} t_{n-k}^{d}\otimes b'.v &= 0.    
 \end{align}
 consider $h' \in \mathfrak{h}$ such that $(h,h')\neq 0$ and applying $h'\otimes t_0^{m+1}t_{n-k}^s\otimes 1$ for any $s \in \mathbb{Z}$ on $ (3.2)$ we have
 
  $$\displaystyle\sum_{d \in \mathbb{Z}} a_d h\otimes t_0^{-m} t_{n-k}^{r-d}\otimes b.h'\otimes t_0^{m+1}t_{n-k}^s\otimes 1            .h\otimes t_0^{-(m+1)} t_{n-k}^{d}\otimes b'.v $$  
  $$+(h,h') \displaystyle\sum_{d \in \mathbb{Z}} a_d(m+1)  t_0 t_{n-k}^{r+s-d}K_0\otimes b.h\otimes t_0^{-(m+1)} t_{n-k}^{d}\otimes b'.v  $$ 
    $$+(h,h') \displaystyle\sum_{d \in \mathbb{Z}} a_d s  t_0 t_{n-k}^{r+s-d} K_{n-k}\otimes b.h\otimes t_0^{-(m+1)} t_{n-k}^{d}\otimes b'.v =0  $$ Now 2nd and 3rd term of the above sum is zero by Theorem \ref{t3.1}(4) and the 1st term imply

   $$\displaystyle\sum_{d \in \mathbb{Z}} a_d h\otimes t_0^{-m} t_{n-k}^{r-d}\otimes b. h\otimes t_0^{-(m+1)} t_{n-k}^{d}\otimes b'.h'\otimes t_0^{m+1}t_{n-k}^s\otimes 1.v $$ 
     $$+(h,h')\displaystyle\sum_{d \in \mathbb{Z}} a_d h\otimes t_0^{-m} t_{n-k}^{r-d}\otimes b.(m+1)t_{n-k}^{s+d}K_0\otimes b'              .v $$ 
      $$+(h,h')\displaystyle\sum_{d \in \mathbb{Z}} a_d h\otimes t_0^{-m} t_{n-k}^{r-d}\otimes b.st_{n-k}^{s+d}K_{n-k}\otimes b'              .v =0.$$ 
1st term of the above sum is zero, since $v$ is a highest weight vector and by Theorem \ref{t3.1}(4) 2nd , 3rd terms are zero, when $s+d \neq 0.$      
 If $s+d = 0$, then the sum is equal to 
$$(h,h') a_{-s} h\otimes t_0^{-m} t_{n-k}^{r+s}\otimes b.(m+1)K_0\otimes b'.v  $$ 
    $$+(h,h') a_{-s} h\otimes t_0^{-m} t_{n-k}^{r+s}\otimes b.sK_{n-k}\otimes b'.v =0 .$$ 
    Now since $k<n$, the last term is zero by Proposition \ref{p3.3} , hence only remaining term is  
    $$ a_{-s} h\otimes t_0^{-m} t_{n-k}^{r+s}\otimes b.(m+1)K_0\otimes b'.v=0 . $$
    Since not all $K_0\otimes b   $ acting on $V$ trivially and they acts as scalar, so to prove $a_{-s}=0$, it is sufficient to prove that $h\otimes t_0^{-m} t_{n-k}^{r+s}\otimes b$ acting on $V $ non trivially for some b.
    \\ Let $h\otimes t_0^{-m} t_{n-k}^{r+s}\otimes b.v=0$, for all b, then 
    $$h'\otimes t_0^m t_{n-k}^{-(r+s)}\otimes 1.h\otimes t_0^{-m} t_{n-k}^{r+s}\otimes b.v=0 $$ which imply 
    $$(h,h')K_0\otimes b.v -(h,h')K_{n-k}\otimes b.v=0 $$ for all b. \\ Thus $K_0\otimes b =0 $ for all b, by Proposition \ref{p3.3}, a contradiction, Hence the claim. \\ Hence we have proved that dimension of the weight space $ V_{\lambda-(2m+1)\delta_0 +r\delta_{n-k}}$ is infinite, a contradiction.

\end{proof} 
\subsection*{4.1}\label{rs4.1}Let $\mathfrak{g}_{af}= \dot{\mathfrak{g}}\otimes{\mathbb{C}[t_0^{\pm 1}]} \oplus{\mathbb{C}K_0} \oplus {\mathbb C}d_0   $ be the affine Lie algebra. For all $\alpha \in \dot{ \Delta}$ and $n \in \mathbb Z$, set

 $$(\mathfrak{g}_{af})_{\alpha +n \delta_1}=\mathfrak{\dot g}\otimes t_0^n,  $$
 $$(\mathfrak{g}_{af})_{ n \delta_1} =\mathfrak{\dot h}\otimes t_0^n, n \neq 0, $$
 $$\Delta_{af}^+=\{ \alpha+n\delta_1: \alpha \in \dot{ \Delta}, n>0\}\cup\{n\delta_1 : n> 0\} \cup \dot{ \Delta}^+, $$
$$\mathfrak{h}_{af}=\mathfrak{\dot h}\oplus K_0\oplus d_0.  $$
Then we have a triangular decomposition of $\mathfrak{g}_{af}$ as, 
  $$\mathfrak{g}_{af}=\mathfrak{g}_{af}^-\oplus \mathfrak{h}_{af} \oplus \mathfrak{g}_{af}^+  ,$$
  where  $\mathfrak{g}_{af}^{+}=\displaystyle{\bigoplus_{\alpha \in \Delta_{af}^+}  (\mathfrak{g}_{af})_{\alpha}}  $ and $\mathfrak{g}_{af}^{-}=\displaystyle{\bigoplus_{\alpha \in \Delta_{af}^+}  (\mathfrak{g}_{af})_{-\alpha}}  $
  \begin{remark}\label{rn4.1}
Now consider the current Kac-Moody Lie algebra $G= \mathfrak{g}_{af}'\otimes{A_n} \otimes B \oplus {\mathbb{C}d_0} .$
Define a surjective Lie algebra homomorphism 
$$\phi :  \dot{ \mathfrak{g}} \otimes A \otimes B \oplus  Z \otimes B \oplus {D}  \to \mathfrak{g}_{af}'\otimes{A_n} \otimes B \oplus {D}     $$

$$X \otimes{t_0^{r_0}t^{\underline r}}\otimes b \mapsto (X \otimes{t_0^{r_0}}) \otimes{ t^{\underline r} } \otimes b,  $$
 $$ t_0^{r_0}t^{\underline r} K_i \otimes b \mapsto 0 , \hspace{.1cm} 1 \leq i \leq n,$$
$$ t_0^{r_0}t^{\underline r} K_0 \otimes b \mapsto 0 , \hspace{.1cm}if \hspace{.1cm} r_0 \neq 0,$$
  $$ t_0^{r_0}t^{\underline r} K_0 \otimes b \mapsto  K_0 \otimes{t^{\underline r}} \otimes b, \hspace{.1cm}if \hspace{.1cm} r_0 = 0,$$
  $$d_i \mapsto d_i , $$
 for all $X \in \dot{\mathfrak{g}}$ , $b \in B$ and ${\underline r}=(r_1,r_2,...,r_n) \in \mathbb{Z}^n$ , $r_0 \in \mathbb{Z}, $ $0 \leq i \leq n$.\\
 Note that from Theorem \ref{t3.1} and \ref{t3.3}, we have the kernel of $\phi$ is acting trivially on the module $V$. Hence $V$ is an irreducible module for $\mathfrak{g}_{af}'\otimes{A_n} \otimes B \oplus {D}    $.
 \end{remark}
 
  By Theorem \ref{t3.1} and \ref{t3.3}, we have non-zero central operators $z_1,...,z_n$ of degree $(0,l_1,...,0),....,(0,...,l_n)$ respectively. Consider the space,
  $$W=span\{ z_nv-v:v \in V\}. $$

\begin{lemma}
$W$ is a proper $\mathfrak{g}_{af}'\otimes{A_n} \otimes B \oplus \{d_0,d_1,...,d_{n-1}\} $ sub-module of $V$. In particular, $V$ is reducible as  $\mathfrak{g}_{af}'\otimes{A_n} \otimes B \oplus \{d_0,d_1,...,d_{n-1}\} $-module
\end{lemma}
\begin{proof}
It is easy to see that $W$ is a $\mathfrak{g}_{af}'\otimes{A_n} \otimes B \oplus \{d_0,d_1,...,d_{n-1}\} $ sub-module of $V$. We show that $W$ cannot contain $\mathfrak{\dot h}\oplus K_0 \oplus D$ weight vector. Let $w$ be a weight vector in $W$. Then, there exists non-zero scalars $c_1,...,c_k$ and distinct weight vectors $v_1,...,v_k$  such that 
  $$w= \displaystyle{\sum_{i=1}^k}c_i(z_nv_i- v_i). $$ 
Note that for each $i$, either $z_nv_i$ have to lie in the weight space of some $v_j(j\neq i)$ or $c_iz_nv_i=w$. But if $c_iz_nv_i=w$ occur for some $i$ then it force that some $c_j$ have to be zero ( since there are $k-1$ pairs $c_iz_nv_i-c_jv_j$, so one $c_jv_j$ will remains). Hence we have, \begin{center}
 $c_mz_nv_m-c_rv_r=w$ for only one pair $(m,r)$ with $m\neq r$ and\\ $c_iz_nv_i=c_jv_j$ for all other pair(i,j). 
\end{center}  
Therefore after a suitable permutation of the vectors $v_i$  we can assume that, $c_iz_nv_i=c_{i+1}v_{i+1}$ for $1 \leq i \leq k-1$ and $c_kz_nv_k-c_1v_1=w$. From this it follows that if $v_1$ has weight $\mu$ then the weight of $z_nv_k$ will be $\mu +kl_n\delta_n$. This implies that either $c_1=0$ or $c_k=0,$ a contradiction. 
  
\end{proof}
 Let $W$ be a non-zero proper  $\mathfrak{g}_{af}'\otimes{A_n} \otimes B \oplus \{d_0,d_1,...,d_{n-1}\} $ sub-module of $V$.
 Let $\overline{\mu}$ be a weight of $V$ with respect to $\mathfrak{\dot h}\oplus K_0 \oplus D$ and let $\mu=\overline{\mu}|_{\mathfrak{\dot h}\oplus K_0 \oplus \{d_0,d_1,...,d_{n-1}\}}$. Let $w\in W$ be a non-zero $\mu$-weight vector. Then,  $w = \displaystyle{\sum_{i=1}^{m}}w_i $ such that $w_i \in V_{\overline{\mu}+k_i\delta_n}$, $k_1<k_2<...<k_m$ and $m\geq 2$, since $W$ cannot contain $\mathfrak{\dot h}\oplus K_0 \oplus D$ weight vector.\\
 
 Define $d(w)=k_m-k_1$ and $d(W)=min\{d(w):w $ is a non-zero $\mu$ weight vector $\}$.  A weight vector $w$ is said to be minimal if $d(w)$ is minimal. Now it will follow from the same proof as in Lemma 3.8 of \cite{22} that $d(W)$ is well defined. Moreover, with the same proof of Lemma 3.9, 3.10 and Proposition 3.11 we can find a maximal proper sub-module $W$ such that $V/W$ is irreducible and have finite dimensional weight spaces with respect to ${\mathfrak{\dot h}\oplus K_0 \oplus \{d_0,d_1,...,d_{n-1}\}}$. Now we consider $W_1=\{z_{n-1}v-v:v \in V/W\}$ and proceed like previously. Continuing this process $n$ times we have the following result. 

\begin{proposition}\label{pn4.4}
Let $V$ be an irreducible module for $\mathfrak{g}_{af}'\otimes{A_n} \otimes B \oplus {D}    $ with finite dimensional weight spaces. Then there exists a maximal proper  $\mathfrak{g}_{af}'\otimes{A_n} \otimes B \oplus {d_0}    $  sub-module $W$ of $V$ such $V/W$ is irreducible and have finite dimensional weight spaces with respect to  $\mathfrak{\dot h}\oplus K_0 \oplus d_0$
\end{proposition}

Recall the definition of highest weight module and classification theorem for irreducible modules of current Kac-Moody Lie algebras from \cite{2}.
\begin{definition}\label{d3.2}
A module $W$ for current Kac-Moody algebra $\mathfrak{g}_B$ is said to be a highest weight module if there exists a vector $w \in W$ such that
\\1. $U({\mathfrak{g}}_B)w = W,$ 
\\2. ${\mathfrak{g}'^+ \otimes B}.w=0,$
\\3. There exists a map $\eta : h'\otimes B \oplus h'' \mapsto \mathbb{C}$ such that $h.w=\eta(h)w$ for all $h \in  h'\otimes B \oplus h''. $ 
\end{definition}
\begin{theorem}\label{tn}(Theorem 3.4, \cite{2})
Let $V$ be an irreducible integrable highest weight module for current Kac-Moody Lie algebra $\mathfrak{g}_B $ with finite dimensional weight spaces. Then $V\simeq \displaystyle{\bigotimes_{i=1}^kV(\lambda_i)}$, for some $k \in \mathbb N,$ where $\lambda_i$'s are dominant integral weights of $\mathfrak g$ and $V(\lambda_i)$'s are irreducible highest weight modules for  $\mathfrak g$.
\end{theorem}

\begin{theorem}\label{t3.4}
Let $V$ be a module for $\tau(B)$ satisfying the conditions of Proposition \ref{p3.1}. Then upto an automorphism of $\tau(B)$ there exists an irreducible integrable highest weight module $\overline V$ for the current Kac Moody algebra $\mathfrak{g}_{af}'\otimes{A_n} \otimes B \oplus {\mathbb{C}d_0}$ with finite dimensional weight spaces such that $\overline V \simeq       
  \displaystyle{\bigotimes_{i=1}^kV(\lambda_i)}$, for some $k \in \mathbb N,$ dominant integral weights  $\lambda_i$ and irreducible modules $V(\lambda_i)$ for affine Lie algebras. 
\end{theorem}
\begin{proof}
By Theorem \ref{t3.1}, Theorem \ref{t3.3} and Proposition \ref{pn4.4}, we have an irreducible module $\overline V=V/W$ for $\mathfrak{g}_{af}'\otimes{A_n} \otimes B \oplus {\mathbb{C}d_0}$ with finite dimensional weight spaces and integrability of $V$ implies integrability for $\mathfrak{g}_{af}'\otimes{A_n} \otimes B \oplus {\mathbb{C}d_0}$. Moreover, if $v_\lambda$ be a highest weight vector for $V$, then $v_\lambda \notin W$, Otherwise $W$ will be whole of $V$. Hence $ \mathfrak{g}_{af}'^+\otimes{A_n} \otimes B.\overline v_\lambda=0$. Note that for all $h \in \mathfrak{\dot h}$, $\underline r \in \mathbb{Z}^n$ and $b \in B$, $h\otimes{t^{\underline r}}\otimes b,$ $ K_0\otimes{t^{\underline r}}\otimes b $ are central in $\mathfrak{g}_{af}'\otimes{A_n} \otimes B \oplus {\mathbb{C}d_0}$. Hence  by irreducibility of $\overline V$, $h\otimes{t^{\underline r}}\otimes b$ and  $ K_0\otimes{t^{\underline r}}\otimes b $ acts as scalar on $\overline V$. Therefore $\overline V$ is an highest weight module for  $\mathfrak{g}_{af}'\otimes{A_n} \otimes B \oplus {\mathbb{C}d_0}$ and hence the result by Theorem \ref{tn}.
\end{proof}

\section{Classification of modules when the center of toroidal Lie algebra acts non-trivially }

Let us consider $\widetilde G =\mathfrak{g}_{af}'\otimes{A_n} \otimes B \oplus {\mathbb{C}d_0} \oplus \widetilde D$, $\widetilde D$ be the space spanned by $d_1, d_2, ..., d_n.$ Then $G=\displaystyle {\bigoplus_{\underline m \in \mathbb{Z}^n }} G_{\underline m}$, where $G_{\underline m}=\{ X \in G: [d_i,X]=m_iX\}$
\\ Now define a  $\widetilde G $ module action on $(\overline V \otimes A_n, \rho(\underline \alpha))$ for any $\underline \alpha =(\alpha_1, \alpha_2,..,\alpha_n) \in \mathbb{C}^n $ by 
\begin{align}
X_r. (\overline v \otimes t^{\underline s})= (X_r. \overline v)\otimes t^{\underline r + \underline s} , X_r \in G_{\underline r}, \\
d_0.(\overline v \otimes t^{\underline s})=(d_0. \overline v)\otimes t^{\underline s} , \\ 
d_i.(\overline v \otimes t^{\underline s})=(\alpha_i+s_i)(\overline v \otimes t^{\underline s}),
\end{align}
for $1\leq i \leq n$. \\  
 
It is easy to see that this action defines a module structure on  $\overline V \otimes A_n$ for any $\underline \alpha \in \mathbb{C}^n .$ But we fix $\underline \alpha$ depending on the situation.
For the rest of the section we fix a highest weight vector $\overline v$ of $\overline V$.
\begin{lemma}\label{l4.1}
Any non-zero $\widetilde G$ sub-module of  $\overline V \otimes A_n$ contains $\overline v \otimes{t^{\underline s}}$ for some $\underline s\in \mathbb{Z}^n$.
\end{lemma}
\begin{proof}
Let $W$ be a non-zero  $\widetilde G$ sub-module of  $\overline V \otimes A_n$. Consider the the map $\phi:\overline V \otimes A_n\to \overline V $ defined by $\phi (w \otimes{ t^{\underline m}})=w$. It is easy to check that $\phi $ is a surjective $G$ module map (but not $\widetilde G$ ).\\
{\bf Claim:} $\phi(W)=\overline V $. Since $\phi$ is surjective and $\overline V $ is irreducible $G$-module so it is sufficient to show $\phi(W) \neq 0 $. 
\\ Let $\overline V =\displaystyle{\bigoplus_{\lambda \in \underline{ h}^*}\overline V_{\lambda}}$ with respect to the Cartan sub-algebra $\underline{ h}=$span $\{\mathfrak h, K_0, d_0 \}$. Then weight spaces of $\overline V \otimes A_n$ are $(\overline V \otimes A_n)_{\lambda + \delta_{\underline r}}= \overline{V_{\lambda}}\otimes {t^{\underline r}} $ with respect to the Cartan, span$ \{\mathfrak h, K_0, d_0 , \widetilde D \} .$ Hence $W $ is also a weight module, so contains a vector of the form $w \otimes{t^{\underline m}} $ for some weight vector $w \in \overline V$, therefore $\phi(W)\neq 0$. 
\\ Now $\overline v \in \overline V$, so there exists a vector $w \in W$ such that $\phi(w)=\overline v$. Since $\phi$ is a $G$ module map, hence $w$ and $\overline v$ has same weight as $G$ module. But $\{ \overline v \otimes {t^{\underline r}} : \underline r \in \mathbb{Z}^n \} $ are the only weights in $\overline V \otimes A_n$ of same weight that of $\overline v$. Hence $w =\displaystyle {\sum_{finite} a_i \overline v \otimes {t^{\underline r_i}}}$. Now $W$ is a $\mathbb{Z}^n$ graded sub-module of $\overline V \otimes A_n$, hence $\overline v \otimes {t^{\underline r_i}} \in W$ for all i. 
\end{proof}  
\begin{remark}\label{rn5.1}
 Consider $\widetilde G$-submodules $U(\widetilde G)\overline v \otimes {t^{\underline r}}$ and $U(\widetilde G)\overline v \otimes {t^{\underline s}}$ and define a map $\phi :U(\widetilde G)\overline v \otimes {t^{\underline r}} \to U(\widetilde G)\overline v \otimes {t^{\underline s}} $ by $\phi(w \otimes {t^{\underline k}})=w \otimes{t^{\underline k + \underline s -\underline r}}$, then $\phi$ is a $G$-module isomorphism $($but not $\widetilde G$-module morphism$)$. Now we define grade shift isomorphism between $U(\widetilde G)\overline v \otimes {t^{\underline r}}$  and $U(\widetilde G)\overline v \otimes {t^{\underline s}}$. Let $\widetilde D$ acts on  $\overline v \otimes {t^{\underline r}}$ by a scalar $\underline\alpha =(\alpha_1, \alpha_2,..., \alpha_n) \in \mathbb{C}^n$, i.e $d_i.\overline v \otimes {t^{\underline r}}=\alpha_i\overline v \otimes {t^{\underline r}}$ and acts on $\overline v \otimes {t^{\underline s}}$ by a scalar $\underline \beta=(\beta_1,\beta_2,..., \beta_n)$. Let $\mathbb{C}$ be a one dimensional module for $\widetilde G$ such that $G.\mathbb{C} = 0$ and $D.\mathbb{C}$ acts as a scalar $\underline \beta -\underline \alpha$. Then clearly $U(\widetilde G)\overline v \otimes {t^{\underline r}} \simeq U(\widetilde G)\overline v \otimes {t^{\underline s}} \otimes {\mathbb{C}}$ as $\widetilde G$ module. we call this as grade shift isomorphism.
\end{remark}

\begin{lemma}\label{ln5.2}
$U(\widetilde G)\overline v \otimes {t^{\underline r}}$ are irreducible $\widetilde G$ module for all $\underline r \in \mathbb{Z}^n$. In fact any irreducible $\widetilde G$-sub-module of $\overline V \otimes {A_n}$ is of the above form. 
\end{lemma}
\begin{proof}
By Remark \ref{rn4.1}, $V$ is an irreducible module for $\widetilde G.$ Therefore $V$ can be written as, $$V=\displaystyle{\bigoplus_{\underline r \in \mathbb{Z}^n}}V_{\underline r},$$ where $V_{\underline r}= \{ v \in V : d_i.v=(\lambda(d_i)+r_i )v $ for $ 1\leq i \leq n \} $ for some weight $\lambda$ of $V$. Define a map $ \psi : V \mapsto \overline V \otimes A_n $ by, 
\begin{center}
$\psi(v_r)=\overline v \otimes t^{\underline r}$ if $v_{\underline r} \in V_{\underline r}.$
\end{center} 

Then $\psi$  is a $\widetilde G$ module map for the fixed choice of $\underline \alpha =(\lambda(\alpha_1),\lambda(\alpha_2),...,\lambda(\alpha_n))\in \mathbb{C}^n $. In fact $\psi $ is non zero, since $\psi(v)=\overline v \otimes 1$ for any $v \in V_\lambda$. Thus $\psi $ is injective, since $V$ is irreducible. Hence $\psi(V)$ is an irreducible sub-module of $\overline V \otimes A_n.$ Hence there exists an irreducible sub-module for $\overline V \otimes {A_n}$.\\
Let $W$ be any irreducible sub-module of  $\overline V \otimes {A_n}$. By  Lemma \ref{l4.1}, $W$ contains $\overline v \otimes {t^{\underline r}}$ for some $\underline r \in \mathbb{Z}^n$, so $U(\widetilde G)\overline v \otimes {t^{\underline r}} = W$. Hence by Remark \ref{rn5.1}, $U(\widetilde G)\overline v \otimes {t^{\underline r}}$ is irreducible for all  $\underline r \in \mathbb{Z}^n$. 
\end{proof}
\begin{lemma}\label{l4.3}  
Each $z$ in $Z\otimes B$ acts as scalar on $\overline V$. Further if $z$ in $Z\otimes B$ acts as non zero on $V$ then it acts as non zero scalar on $\overline V $ 
\end{lemma}
\begin{proof}
Follows from Lemma 3.1 of \cite{5}.
\end{proof}  
Consider non-zero central operators $z_1,...,z_n$ of degree $(0,l_1,...,0),....,(0,...,l_n)$ respectively. Let $\underline{l_i}=(0,0,..,l_i,..0)  $ and $z_i$ acts on $\overline V$ by $l_i'.$ Then by Lemma \ref{l4.3}, $l_i'\neq 0.$ It is easy to see that, we have $z_i(w \otimes {t^{\underline m}})=l_i'(w\otimes{t^{{\underline m+\underline l_i}}} )$, this implies that $z_i$ is an invertible central operator on  $\overline V \otimes {A_n}$. Let 
     $$\Gamma=l_1 \mathbb{Z}+ l_2 \mathbb{Z}+ .....+ l_n \mathbb{Z}. $$  
Then for all $\underline s =(l_1s_1,..,l_ns_n)\in \Gamma,$ there exists an invertible central operator $ z_{\underline s} = \prod_{i=1}^{n}z_i^{s_i}$ on $\overline V \otimes A_n $. Note that $z_{\underline s} (U(\widetilde G)\overline v \otimes {t^{\underline r}})\subseteq (U(\widetilde G)\overline v \otimes {t^{\underline r + \underline s}}) $ and hence  $z_{\underline s}(U(\widetilde G)\overline v \otimes {t^{\underline r}})= (U(\widetilde G)\overline v \otimes {t^{\underline r + \underline s}})$ .\\

 Let $S_0$ be the space spanned by all central operators of the form $z_{\underline s}$ for all $\underline s \in \Gamma.$ Also let $U( G) = \displaystyle{\bigoplus_{\underline m \in \mathbb{Z}^n}}U(G)_{\underline m} $, where $U(G)_{\underline m }= \{ X \in U(G) : [d_i,X]=m_iX , 1\leq i \leq n \}$. 
\begin{lemma}(Theorem 3.5,\cite{17})\label{ln5.4}
 Let $M$ be a module for an associative ring $R$ with unity.
Then the following statements are equivalent:\\
1. $M$ is the sum of a family of irreducible sub-modules.\\
2. $M$ is the direct sum of a family of irreducible sub-modules.\\
3. Every sub-module of $M$ is a direct summand.
\end{lemma}
\begin{proposition}\label{p4.1}
1. $\overline V \otimes A_n=\displaystyle {\sum_{\underline s \in \mathbb{Z}^n} (U(\widetilde G)\overline v \otimes {t^{\underline s}})} = \displaystyle{\bigoplus_{\underline s \in F} (U(\widetilde G)\overline v \otimes {t^{\underline s}})}$, where $F\subseteq \{(s_1,s_2,..,s_n) : 0\leq s_i <l_i \} $ as $\widetilde G $ module. \\
2. $V \simeq (U(\widetilde G)\overline v \otimes 1)$ as $\widetilde G $ module.
\end{proposition}
\begin{proof}
1. Let $w \otimes {t^{\underline s}}\in \overline V \otimes A_n $, since $\overline V$ is irreducible there exist $X \in U(G)$ such that $X.\overline v= w $. Let $X=\sum {X_{\underline r}}$, $X_{\underline r} \in U( G)_{\underline r} $. Then $$\sum {X_{\underline r}.(\overline v \otimes{t^{{\underline s-\underline r}}}  )}=(\sum {X_{\underline r}}.\overline v )\otimes t^{\underline s} =  w \otimes {t^{\underline s}}, $$ hence $\overline V \otimes A_n=\displaystyle {\sum_{\underline s \in \mathbb{Z}^n} (U(\widetilde G)\overline v \otimes {t^{\underline s}})}$. \\
Since $z_{\underline s}(U(\widetilde G)\overline v \otimes {t^{\underline r}})= U(\widetilde G)\overline v \otimes {t^{\underline r + \underline s}}$ for all  $\underline s \in \Gamma$, we have $\overline V \otimes A_n=\displaystyle {\sum_{0 \leq \underline s_i < l_i} (U(\widetilde G)\overline v \otimes {t^{\underline s_i}})}$ as $U(\widetilde G)\oplus S_0$ module.
Now using Lemma \ref{ln5.4}, we have $\overline V \otimes A_n = \displaystyle {\bigoplus _{\underline s \in F} U(\widetilde G))\overline v \otimes {t^{\underline s}}}$  as $U(\widetilde G)\oplus S_0$ module, where $F\subseteq \{(s_1,s_2,..,s_n) : 0\leq s_i <l_i \} $. But all $U(\widetilde G))\overline v \otimes {t^{\underline s_i}}$ are irreducible $\widetilde G$ modules by Lemma \ref{ln5.2}. Hence the result.\\  

2. Proceeding as in the proof of Lemma \ref{ln5.2}, $\psi(V)$ is irreducible sub-module of $\overline V \otimes A_n$ containing $\overline v \otimes 1.$ Again,  $\overline V \otimes A_n$ is direct sum of finitely many irreducible $\widetilde G$ modules. So $\psi(V)=(U(\widetilde G)\overline v \otimes 1)$, by Lemma \ref{ln5.4}. Therefore $V\simeq (U(\widetilde G)\overline v \otimes 1), $ as $\psi$ is injective.
\end{proof}
\begin{theorem}\label{t4.1}
Let $V$ be an irreducible integrable representation of  $\tau(B)$ with finite dimensional weight spaces with respect to $H$ such that some of $K_i$'s acts non-trivially on $V$. Then upto an automorphism of $\tau(B)$,\\ $(i)$ $V$ is isomorphic to an irreducible component of $(\overline V \otimes A_n, \rho(\underline \alpha))$ for some fixed $\underline \alpha \in \mathbb{C}^n$.\\ 
$(ii)$ $\overline V \simeq \displaystyle{\bigotimes_{i=1}^{k}V(\lambda_i)} $ for some $k \in \mathbb N,$ dominant integral weights $\lambda_i$, and irreducible modules $V(\lambda_i)$ for affine lie algebras, for $1\leq i \leq k.$ 
\end{theorem}
\begin{proof}
Follows from Proposition \ref{p4.1} and Theorem \ref{t3.4}.
\end{proof}
\section{Classification of modules when the center of toroidal Lie algebra acts trivially}
In this section we will classify irreducible integrable modules for $\tau(B)$ when $K_i,$ for $0 \leq i \leq n$ acts trivially on modules .
\begin{proposition}\label{p5.1}
Let $V$ be an irreducible integrable module for $\tau(B)$ with finite dimensional weight spaces with respect to $H$. Suppose all $K_i$ $( 0\leq i \leq n )$ acting trivially on $V$. Then there exists weight vectors $v_0 , w_0$ such that $$\dot{\mathfrak{g}}^{+}\otimes A \otimes B. v_0=0 $$ and 
  $$\dot{\mathfrak{g}}^{-} \otimes A \otimes B. w_0=0 .$$
  In particular, if $v_0 \in V_\lambda$, then $\lambda(\alpha^\vee)\geq 0$ for all $\alpha\in \dot \Delta^+$.
\end{proposition}
\begin{proof}
See Proposition 2.12 of \cite{1} and Theorem 2.4(ii) of \cite{11}. The same proof will work. Last statement follows from Lemma \ref{l2.3}(4).
\end{proof}
\begin{lemma}\label{ln6.1}
Let $v_0$ be a weight vector of weight $\lambda$ such that $\dot{\mathfrak{g}}^{+}\otimes A \otimes B. v_0=0 $. Then either dim$V$=1 or $\lambda|_{\mathfrak{\dot h}}\neq0$.
\end{lemma}
\begin{proof}
Let $\lambda|_{\mathfrak{\dot h}}=0$, then proceeding similarly like proof of Lemma \ref{ln4.4}(1), we have $\mathfrak{\dot g}\otimes A\otimes B$ acts trivially on $v_0$. Now consider $W=\{v\in V:\mathfrak{\dot g}\otimes A\otimes B.v=0\},$ which  is a $\tau(B)$ sub-module of $V$. Hence by irreducibility $W=V$. Moreover, $Z\otimes B$ is generated by $\mathfrak{\dot g}\otimes A\otimes B$ as Lie algebra, hence $V$ is an irreducible module for the abelian Lie algebra $D$.
\end{proof}

\begin{proposition}\label{p5.2}
Let $V$ be a non-trivial module for $\tau(B)$ satisfying the conditions of Proposition  \ref{p5.1}. Then $Z \otimes B$ acts on $V$ trivially.
\end{proposition}
\begin{proof}
Let $t_0^{r_0}t^{\underline r}K_i\otimes b$ acting on $V$ non trivially for some $r_0 \in \mathbb Z$ and $\underline r \in \mathbb{Z}^n$. Then by Theorem \ref{t3.1}(3,4), $i \leq n-k$ and $r_j=0 $ for $0 \leq j \leq n-k$.
\\  Fix  some $i \leq n-k$ , $(r_0, \underline r) \in \mathbb{Z}^{n+1}$, $ b\in B$  and $h,h' \in \dot{ \mathfrak{h}} $ such that $(h,h') \neq 0.$
Let us construct the Heisenberg algebra $\widetilde H = $ span $\{ h\otimes t_0^{r_0}t^{\underline r}t_i^k \otimes b, h'\otimes t_i^{-k} , t_0^{r_0}t^{\underline r}K_i \otimes b ,$ $ k \in \mathbb{Z}\} $  with a Lie bracket defined by 
 $$ \left[h\otimes t_0^{r_0}t^{\underline r}t_i^k \otimes b, h'\otimes t_i^{-l} \right]= k(h,h')t_0^{r_0}t^{\underline r}K_i \otimes b \delta_{kl}.$$
By Proposition \ref{p5.1}, there exists a weight vector $v_\lambda$ of weight $\lambda$ such that   $\dot{\mathfrak{g}}^{+}\otimes A \otimes B. v_{\lambda}=0 .$ 
 Now by Theorem \ref{t3.1}, there is a proper submodule $W$ of $V$ such that $v_\lambda \not \in W $. Consider $M=U(\widetilde H)\overline{v_\lambda}$ , $\overline{v_\lambda} \in \overline V= V/W$. Then $M$ is a $\mathbb Z$ graded $\widetilde H$ module, in fact $M=\displaystyle{\bigoplus_{k\in \mathbb{Z}}}M\cap {\overline V}_{\lambda +r_0\delta_0 + \delta_{\underline r} +k\delta_i} $ , hence each component of $M$ is finite dimensional, since weight spaces of $\overline V $ are so.\\
 Now since $\lambda$ is dominant and by Lemma \ref{ln6.1}, $\lambda|_{\mathfrak{\dot h}}\neq0$, so by Lemma \ref{ln4.4} there exists a $\omega \in \Omega$ such that $\omega(\lambda + k \delta_i)=\lambda + \overline k \delta_i$, where $0 \leq \overline k < R$ for some $R \in \mathbb{N}$. By Lemma \ref{l2.3}(2) $\{dim {\overline V}_ {\lambda +r_0\delta_0 + \delta_{\underline r} +k\delta_i} : k \in \mathbb{Z}\}$  is a finite set and hence dimensions of components of $M$ are uniformly bounded, which is not possible by Proposition \ref{p3.2}.
\end{proof}
 
We fix some notations for the rest of this section. Let $\tau_1 =\dot{ \mathfrak{g}}\otimes A \otimes B$,  $\widetilde \tau_1 = \tau_1 \oplus D $, ${\mathfrak{h}_1 }=\dot{ \mathfrak{h}} \otimes A \otimes B  $ and $\widetilde {\mathfrak{h}}_1 = \mathfrak{h}_1 \oplus D$, then $U(\tau_1)$ and $U({\mathfrak{h}_1 })$ are both $\mathbb{Z}^{n+1}$ graded algebras.
Let $\overline \psi : U( {\mathfrak{h}_1})\mapsto A  $ be a $\mathbb{Z}^{n+1}$ graded algebra homomorphism and $A_{\overline \psi} =$ Image of $\overline \psi$.
We make $A_{\overline \psi}$ into a $ {\widetilde{\mathfrak h}_1 }$ module by defining 
\begin{align}\label{a5.1}
 h\otimes t^{\underline r} \otimes b.t^{\underline m}=\overline \psi ( h\otimes t^{\underline r} \otimes b) t^{\underline m}
 \end{align}
 \begin{align}\label{a5.2}
  d_i.t^{\underline m} =(m_i+\alpha_i)t^{\underline m},
\end{align}           
    
 for arbitrary $\underline \alpha =(\alpha_1,\alpha_2,...,\alpha_{n+1}) \in \mathbb C^{n+1}$ and for all $h \in \dot{ \mathfrak{h}}$, $b \in B $, $\underline r , \underline m \in \mathbb{Z}^{n+1}.$ But we fix $\underline \alpha$ depending on situation.
\begin{lemma}\label{l5.1}$($Lemma 1.2 , \cite{8}$)$
$A_{\overline \psi}$ is an irreducible $\widetilde{\mathfrak h}_1$ module iff every homogeneous elements of $A_{\overline \psi}$ is invertible in $A_{\overline \psi}$.
\end{lemma}

Let $\overline \psi $ be as above such that $A_{\overline \psi}$ is irreducible and let $\dot{\mathfrak g}^+\otimes A \otimes B$ acting on $A_{\overline \psi}$ trivially. Now construct the induced $\widetilde \tau_1$ module $M(\overline \psi)=U(\widetilde \tau_1) {\displaystyle{\otimes_{U(\dot{\mathfrak g}^+\otimes A \otimes B\oplus \widetilde{\mathfrak h}_1 ) }} } A_{\overline \psi}$. By standard arguments we have
\begin{proposition}\label{p5.3}
1. As a $\widetilde {\mathfrak h}_1$ module $M(\overline \psi)$ is a weight module.\\

2.$M(\overline \psi)$ is a free  $\dot{\mathfrak g}^-\otimes A \otimes B$ module and as a vector space $M(\overline \psi) \simeq U(\dot{\mathfrak g}^-\otimes A \otimes B)\otimes \mathbb{C}$.\\

3.$M(\overline \psi)$ has a unique irreducible quotient, say $V(\overline \psi )$.
\end{proposition}
A module $V$ for $\widetilde \tau_1$ is said to be a graded highest weight module if there exists a weight vector $v$ such that\\
 (1) $U(\widetilde \tau_1).v=V.$\\
(2)  $\dot{\mathfrak g}^+\otimes A \otimes B.v=0$.\\
(3) $U(\widetilde {\mathfrak h}_1)v$ is an irreducible-$\widetilde{\mathfrak h}_1$ module isomorphic to $ A_{\overline \psi}$ for some $\mathbb{Z}^{n+1}$- graded homomorphism $\overline \psi$.\\

A module $W$ is said to be a non graded highest weight module for $\tau_1$ if there exists a weight vector $w$ such that \\
 (1) $U(\tau_1).w=W.$ \\
 (2)  $\dot{\mathfrak g}^+\otimes A \otimes B.w=0.$\\
  (3) there exists $\psi \in \mathfrak{h}_1^*$ such that $h.w=\psi(h)w$ for all $h \in \mathfrak{h}_1^*$. \\

Let $\psi \in \mathfrak{h}_1^* $ and $\mathbb{C}\psi$ be a one dimensional vector space. Let  $\dot{\mathfrak g}^+\otimes A \otimes B$ acting on  $\mathbb{C}\psi$ trivially and $\mathfrak{h}_1$ acting by $\psi$. Now construct the induced module for $\tau_1$, $M( \psi)=U( \tau_1) {\displaystyle{\otimes_{U(\dot{\mathfrak g}^+\otimes A \otimes B\oplus {\mathfrak h}_1 ) }} } {\mathbb{C} \psi}$. By standard arguments we have 
\begin{proposition}
1. As a $ {\mathfrak h}_1$-module, $M( \psi)$ is a weight module.\\

2.$M( \psi)$ is a free  $\dot{\mathfrak g}^-\otimes A \otimes B$ module and as a vector space $M(\psi) \simeq U(\dot{\mathfrak g}^-\otimes A \otimes B)\otimes \mathbb{C}$.\\

3.$M( \psi)$ has a unique irreducible quotient, say $V( \psi )$.
\end{proposition}

Let $\overline \psi : U({\mathfrak h}_1) \mapsto A_{\overline \psi}$ be a $\mathbb{Z}^{n+1}$ graded homomorphism and let $E(1): A_{\overline \psi} \mapsto \mathbb{C}$ be a linear map such that $E(1)t^{\underline r}=1 $. Let $\psi =E(1) \overline{\psi} $ and $V(\psi)$ be the irreducible $\tau_1$ module. We will make $V(\psi)\otimes A$ into a $\widetilde \tau_1$ module by the action, 
\begin{align}
x \otimes t^{\underline r}\otimes b.(v\otimes t^{\underline s}) = (x\otimes t^{\underline r}\otimes b .v) \otimes{ t^{\underline r + \underline s}},\\
        d_i.(v \otimes t^{\underline s}) = s_i(v \otimes t^{\underline s}) ,
\end{align}
  for all $\underline r , \underline s \in \mathbb{Z}^{n+1}$, $b \in B$.

\begin{proposition}\label{p5.5}
Let $\overline \psi$,  $\psi$ be as above and $A_{\overline \psi}$ be an irreducible $\widetilde{\mathfrak h}_1$ module. Let $G \subset \mathbb{Z}^{n+1}$ and $\{ t^{\underline m} : \underline m \in G  \}$ denote the set of coset representative of $A/A_{\overline \psi}$. Let            $v\in V(\psi)$ be the highest weight vector. Then \\
1. $V(\psi)\otimes A = \displaystyle{\bigoplus_{\underline m \in G}  }  U(v\otimes t^{\underline m})$ as a $\widetilde \tau_1$ module, where $U(v\otimes t^{\underline m})$ is  $\widetilde \tau_1$ submodule generated by $v\otimes t^{\underline m}$.\\
2. Each $U(v\otimes t^{\underline m})$ is irreducible $\widetilde \tau_1$ module.\\
3. $V(\overline \psi) \simeq U(v\otimes 1)$ as $\widetilde \tau_1$ module.\\ 
4. $V(\overline \psi) $ has finite dimensional weight spaces as $\widetilde{\mathfrak h}_1$ module iff $V(\psi)$ has finite dimensional weight spaces as $\mathfrak{h}_1$ module .
\begin{proof}
Proof follows on the similar lines as in Proposition 1.8, 1.9 and Lemma 1.10 of \cite{9}.
\end{proof}
 
\end{proposition}
\begin{lemma}\label{l5.2}
Let $I$ and $J$ be two co-finite ideal of a commutative, associative, finitely generated algebra $S$ over $\mathbb C$. Then $IJ$ is a co-finite ideal of $S$. In particular, product of finite numbers of co-finite ideals of $S$ is co-finite.
\end{lemma}
\begin{proof}
To prove this we use the fact that a finitely generated algebra over $\mathbb C$ is Artinian iff it is a finite dimensional over $\mathbb C$, \cite{14} (Chapter 8, Exercise 3).
Now, dim$S/I$ and dim$S/J$ both are finite so both are Artinian ring. \\
Consider the homomorphism $f: S/IJ \to S/I \times S/J$ by $f(a+IJ)=(a+I,a+J)$, then $f(S/IJ)$ is finitely generated as an algebra over $\mathbb C$ and is of finite dimensional. Hence $f(S/IJ)$ is an Artinian ring.\\
Now consider an exact sequence $0 \to S/I \to S/IJ \to f(S/IJ) \to 0$, since both $S/I$ and $f(S/IJ)$ are Artinian, so is  $S/IJ$. 

\end{proof}

\begin{proposition}\label{p5.6}
Let $V$ be an irreducible module for $\dot{\mathfrak{g}} \otimes S $ with finite dimensional weight space with respect to  $\dot{\mathfrak{h}}$, where $S$ is a commutative associative finitely generated unital algebra over $\mathbb{C}$. Then there exists a co-finite ideal $I$ of $S$ such that $V$ is a module for a finite dimensional algebra $\dot{\mathfrak{g}} \otimes S/I $ .
\end{proposition}
\begin{proof}
Let $V_\lambda$ be a weight spaces of $V$. Then $V_\lambda $ is invariant under the commutative sub-algebra $\dot{\mathfrak{h}}\otimes S$ of $\dot{\mathfrak{g}}\otimes S$. Hence by Lie's theorem there exists a non zero vector $v \in V_\lambda$ such that $h\otimes s.v=\mu(h\otimes s)v$, for some $\mu \in (\dot{\mathfrak h}\otimes S)^*$, for all $h \in \dot{\mathfrak h}$, $s \in S$.\\ 
Fix a root $\alpha $ of $\dot{\mathfrak g}$ and let $X_\alpha$ be the root vector corresponding to the root $\alpha$.
Let $I_\alpha= \{s\in S: X_\alpha \otimes s.v=0  \}$\\
{\bf Claim:} $I_\alpha$ is an ideal of $S$.\\
Let $s \in I_\alpha $ and $t\in S$, then $X_\alpha \otimes s.v=0$. Now consider $h \in \mathfrak{\dot h} $ such that $\alpha(h)\neq 0$.\\
Then 

$$h\otimes t . X_\alpha \otimes s.v = 0, $$
 $$ \Rightarrow \left[h\otimes t , X_\alpha \otimes s\right].v + X_\alpha \otimes s . h\otimes t.v =0 ,$$
 $$  \Rightarrow \left[h\otimes t , X_\alpha \otimes s\right].v + \mu( h\otimes t) X_\alpha \otimes s .v =0,  $$
 $$  \Rightarrow \left[h , X_\alpha \right]\otimes st.v=0, $$ 
 i.e $ \alpha{(h)} X_\alpha \otimes{ st}.v=0$, hence $st \in I_\alpha$, since $\alpha(h) \neq 0.$\\
{\bf Claim :} dim$S/I_\alpha$ is finite.\\
Let dim$V_{\lambda + \alpha} \leq k$. Now consider $n (> k)$ non zero elements in $S/I_\alpha$, say $s_1+I_\alpha,s_2+I_\alpha,...,s_n+I_\alpha$.\\
Consider the vectors $X_\alpha\otimes s_1.v,X_\alpha\otimes s_2.v,...., X_\alpha\otimes s_n.v$ in $V_{\lambda + \alpha}. $ Since dim$V_{\lambda +\alpha }$ is $ \leq k$ so, there exists $c_1,c_2,...,c_n \in \mathbb C$ not all zero such that 
$$X_\alpha \otimes {\displaystyle{\sum_{i=1}^{n}}c_is_i.v}=0, $$ hence $\displaystyle{\sum_{i=1}^{n}}c_is_i \in I_\alpha$. So any set of $n$ non zero vectors in $S/I_\alpha$ are linearly dependent. Hence dim$S/I_\alpha \leq k.$ \\
Now consider $I=\displaystyle{\prod_{\alpha \in \dot{\Delta}}I_\alpha}$, then $S/I$ is finite dimensional, by Lemma \ref{l5.2}.\\
{\bf Claim:} $\mathfrak{\dot g}\otimes I .v=0$\\
Clearly $\mathfrak{\dot g}^{\pm}\otimes I .v =0 $, since $I \subseteq I_\alpha$ for all $\alpha \in \dot{\Delta.}$  Again, $h_{\alpha}\otimes I.v=\left[ X_\alpha, X_{-\alpha} \right]\otimes I.v=0$ for all $\alpha \in \dot{\Delta,}$ Hence the claim. 
\\
Now consider $W=\{v\in V :\dot{\mathfrak g}\otimes I.v=0  \}$, which is a non zero sub-module of $V$, hence by irreducibility $W=V$. Thus $V$ is a module for $\dot{\mathfrak g}\otimes S/I$.

\end{proof}

\begin{lemma}\label{l5.3}$($Lemma 5.1, \cite{1}$)$
For $n \geq 1,$ any $\mathbb{Z}^{n}$ graded simple, commutative algebra of which each graded component is finite dimensional is isomorphic to a subalgebra of $A_{n}$ such that each homogeneous element is invertible.
\end{lemma}
\begin{theorem}\label{t5.1}
Let $V$ be an irreducible integrable module for $\tau(B)$ with finite dimensional weight spaces with respect to $H$. Then $V \simeq V(\overline \psi)$ for some $\mathbb{Z}^{n+1}$ graded homomorphism $\overline \psi :U(\mathfrak{h}_1) \mapsto A $ such that $A_{\overline \psi}$ is irreducible $\widetilde {\mathfrak{h}_1}$-module. Moreover $V(\psi) $ is finite dimensional module for $\tau_1$, where $\psi =E(1)\overline \psi$.
\end{theorem}
\begin{proof}
By Proposition \ref{p5.1}, $V$ is actually a module for $\widetilde \tau_1$. Now consider the $\widetilde{\mathfrak h}_1$ module $W=U(\widetilde{\mathfrak h}_1)v$ for some weight vector $v$ of weight $\lambda$. Then 
$$ W=U(\widetilde{\mathfrak h}_1)v = U({\mathfrak h}_1)v =  \displaystyle{\bigoplus_{\underline m \in \mathbb{Z}^{n+1} }}V_{\lambda + \delta_{\underline m }}\cap W.$$  Therefore each graded component of $W$ is finite dimensional. \\ 
Define $\phi : U({\mathfrak h}_1) \mapsto W $ by $X_{\underline r} \mapsto X_{\underline r}v$ for $X_{\underline r} \in   U({\mathfrak h}_1)_{{\underline r} }$, where $ U({\mathfrak h}_1)_{{\underline r} }$ is the $\underline r$-th component of $ U({\mathfrak h}_1)$. Clearly $\phi$ is a $\mathbb{Z}^{n+1}$ graded surjective $\mathfrak {h}_1$ module  morphism. Thus $ U({\mathfrak h}_1)/ker \phi \simeq W$. Now irreducibility of $V$ imply that $W$ is irreducible $\widetilde{\mathfrak h}_1$ module. Hence $ U({\mathfrak h}_1)/ker \phi$ is a commutative, associative graded simple algebra, since $U({\mathfrak h}_1) $ is commutative. Now it follows from Lemma \ref{l5.3}, that there exists a $\mathbb{Z}^{n+1}$ graded injective homomorphism $\eta: U({\mathfrak h}_1)/ker \phi \to A.$ Let $\eta_1:U({\mathfrak h}_1) \to U({\mathfrak h}_1)/ker \phi$ be the canonical map and consider $\overline \psi= \eta\circ \eta_1.$ Then $\eta$ becomes an $\mathfrak{h}_1$-module isomorphism between $U({\mathfrak h}_1)/ker \phi$ and $A_{\overline \psi}$ and hence $W\simeq A_{\overline \psi}$ as $\mathfrak{h}_1$-module. Now in the action of (\ref{a5.2}), for fixed choice of $\underline \alpha=(\lambda(d_1),\lambda(d_2),....,\lambda(d_{n+1}))$ this isomorphism extend to an $\widetilde{\mathfrak{h}}_1$-module isomorphism between $W$ and $A_{\overline \psi}$. Hence $V \simeq V(\overline \psi).$\\
Clearly $V(\psi)$ is an irreducible integrable module for $\tau_1$. Again by Proposition \ref{p5.5}, $V(\psi)$ has finite dimensional weight spaces with respect to $\mathfrak h_1$. Now by Proposition \ref{p5.6}, there exists a co-finite ideal $I$ of $A \otimes B$ such that $V(\psi)$ is a module for $\mathfrak{\dot g} \otimes( A\otimes B/I).$ Now consider the highest weight vector $v \in V(\psi)$, then by PBW theorem $V(\psi)=U(\mathfrak{\dot g}^-\otimes( A\otimes B/I) )v. $ But each vector of $ \mathfrak{\dot g}^-\otimes( A\otimes B/I) $ acts locally nilpotently on $v$. Hence $V(\psi)$ is finite dimensional. 

\end{proof}

Let $S$ be a finitely generated commutative associative unital algebra over $\mathbb{C}$. Let $Max S$ denote the set of all maximal ideals of $S$. Let $Supp(Max S, \dot P^+)$ denote the set of all finitely supported functions from $Max S$ to $\dot P^+$ and $|Supp(Max S,\dot P^+)|$ be its cardinality. 
\begin{theorem}\label{5.2}
Let S be an associative commutative finitely generated unital algebra over $\mathbb{C}$. Then any finite dimensional irreducible modules for $\dot{\mathfrak{g}}\otimes S $ reduces to a modules for $\displaystyle{\bigoplus_{N_0-copies}\mathfrak{\dot g}} $ $(N_0 = | Supp(Max S,\dot P^+)|) $ isomorphic to
 $\displaystyle{\bigotimes_{\chi(M)\neq 0}} V(\chi(M))$ ,
 for some $\chi \in Supp(Max S,\dot P^+)$,      
\end{theorem}
\begin{proof}
Proposition in section 6 of \cite{10}.  
\end{proof}
\begin{theorem}\label{t5.3}
Let $V$ be an irreducible integrable representation of $\tau(B)$ with finite dimensional weight spaces with respect to $H$ such that all $K_i$'s acts trivially on $V$. Then \\
$(i)$ there exists a $\mathbb{Z}^{n+1}$ graded morphism $\overline \psi :U(\mathfrak{h}_1) \mapsto A $ such that $V\simeq V(\overline \psi)$.\\
$(ii)$ $V(\overline \psi) $ is isomorphic to an irreducible component of $V(\psi) \otimes A$ and $V(\psi)$ is a finite dimensional irreducible module for  $\displaystyle{\bigoplus_{N_0-copies}\mathfrak{\dot g}} $ $(N_0 = | Supp(Max A \otimes B,\dot P^+)|) $  which is isomorphic to  $\displaystyle{\bigotimes_{\chi(M)\neq 0}} V(\chi(M)),$ 
 for some $\chi \in Supp(Max A\otimes B, \dot P^+)$.       
 
\end{theorem}
\begin{proof}
Proof follows from Theorem \ref{5.2}, Theorem \ref{t5.1} and Proposition \ref{p5.5}.
\end{proof}

\vspace{1cm}

{\bf Acknowledgments:}
We would like to thank Professor S. Eswara Rao for  helpful discussion.

\end{document}